\documentclass[11pt]{amsart}
\usepackage{amscd,amssymb,verbatim}

\theoremstyle{plain}
\newtheorem{thm}{Theorem}
\newtheorem{cor}[thm]{Corollary}
\newtheorem{lem}[thm]{Lemma}
\newtheorem{prop}[thm]{Proposition}

\newcommand{\ep}{\varepsilon}
\newcommand{\al}{\alpha}

\newcommand{\bs}{\backslash}

\newcommand{\ol}{\overline}

\newcommand{\diam}{\operatorname{diam}}

\newcommand{\Fac}{\operatorname{Fac}}

\newcommand{\N}{{\mathbb N}}

\newcommand{\E}{{\mathbb E}}

\newcommand{\G}{{\mathcal G}}

\begin{document}

\title{Ideals of operators on $(\oplus \ell^\infty(n))_{\ell^1}$}

\begin{abstract}
The unique maximal ideal in the Banach algebra $L(E)$, $E = (\oplus \ell^\infty(n))_{\ell^1}$, is identified.  The proof relies on techniques developed by Laustsen, Loy and Read \cite{LLR} and a dichotomy result for operators mapping into $L^1$ due to Laustsen, Odell, Schlumprecht and Zs\'{a}k \cite{LOSZ}.
\end{abstract}

\author{Denny H.\ Leung}
\address{Department of Mathematics, National University of Singapore, Singapore 119076}
\email{matlhh@nus.edu.sg}

\thanks{Research of the author was partially supported by AcRF project no.\ R-146-000-157-112}

\maketitle

Given a Banach space $E$, it is a natural problem to try to understand the ideal structure of the Banach algebra $L(E)$. In the classical period, the only Banach spaces $E$ for which the lattice of closed ideals in $L(E)$ are fully understood are Hilbert space \cite{C, G,L} and the sequence spaces $c_0$ and $\ell^p$, $1\leq p < \infty$ \cite {GMF}.  In the past decade, starting with \cite{LLR}, there has been a resurgence of interest in the problem and new results have been obtained, see, e.g., \cite{KPSTT, KL, LOSZ, LSZ, LinSZ, SST, S}.
One should also mention the recent breakthrough example $AH$ by Argyros and Haydon \cite{AH}, where $L(AH)$ is well understood because of  a very different reason. 
In this note, we make a small contribution to the program by identifying the unique maximal ideal in the algebra $L(E)$, where $E = (\oplus \ell^\infty(n))_{\ell^1}$.
Our method is slightly more general in the sense that it works also for the space $E = (\oplus \ell^2(n))_{\ell^1}$, which gives an alternate proof of the main result of \cite{LSZ}.

Let $E_n$ be finite dimensional Banach spaces and let $E = (\oplus E_n)_{\ell^1}$.  Given a subset $I$ of $\N$, denote by $E_I$ the subspace $(\oplus_{n\in I}E_n)_{\ell^1}$.  The natural embedding of $E_I$ into $E$ is denoted by $J_I$ and the natural projection from $E$ onto $E_I$ is denoted by $P_I$. If $I = \{n\}$, then we write for short $J_n$ and $P_n$ respectively. Let $T$ be a bounded linear operator on $E$, then $T$ has a natural matrix representation $(T_{mn})$, where $T_{mn} = P_mTJ_n: E_n \to E_m$.  Denote the $n$th column in this representation by $T_n$. Thus $T_n = TJ_n : E_n \to E$.  Since $E_n$ is finite dimensional, $\lim_m J_{[m,\infty)}P_{[m,\infty)}T_n = 0$ in the operator norm.

Let $X, Y, Z$ be  Banach spaces. If $T: X\to Y$ is a bounded linear operator and $\ep > 0$, define
\[\Fac^\ep_Z(T) = \inf\{\|R\|\,  \|S\|: R\in L(X,Z), S\in L(Z,Y), \|T - SR\| \leq \ep\}.\]
For $Z = c_0$, respectively $\ell^1$, we write $\Fac^\ep_0(T)$ and $\Fac^\ep_1(T)$, respectively.
Let
\[ \G_Z(X,Y) = \{T\in L(X,Y): T= SR, R\in L(X,Z), S\in L(Z,Y)\}.\]
In the sequel, $T$ will always denote a bounded linear operator on $E$.

\begin{lem}\label{lem1}
Suppose that there exists an embedding $\theta: E \to \ell^\infty$ such that $\sup_n\Fac^\ep_1(\theta T_n) < \infty$.  Then $d(\theta T, \G_{\ell^1}(E,\ell^\infty))\leq \ep$, where $d$ is the distance in the operator norm.
\end{lem}

\begin{proof}
There are a finite constant $C > 0$ and operators $R_n : E_n\to \ell^1$, $S_n: \ell^1 \to \ell^\infty$ such that $\|\theta T_n - S_nR_n\|\leq \ep$, $\|R_n\| \leq 1$, $\|S_n\| \leq C$ for all $n$.  For $x = (x_n) \in E$, let $Rx = (R_nx_n) \in (\oplus \ell^1)_{\ell^1}$. Define $S: (\oplus \ell^1)_{\ell^1} \to \ell^\infty$ by $Sy = \sum S_ny_n$, where $y = (y_n)$. It is easy to check that $\|R\| \leq 1$, $\|S\| \leq C$, and that $\|\theta T - SR\| \leq \ep$. Since $(\oplus \ell^1)_{\ell^1}$ is isometric to $\ell^1$, $SR \in \G_{\ell^1}(E,\ell^\infty)$.
\end{proof}

\begin{lem}\label{lem2}
Suppose that $W = \sum^3_{i=1}W_i$, where $W, W_i$ belong to $L(X,Y)$. Then, for any $\ep > 0$, $\Fac^{3\ep}_1(W) \leq 3\max_{1\leq i \leq 3}\Fac^\ep_1(W_i)$.
\end{lem}

\begin{proof}
For $1\leq i\leq 3$ and any $\delta > 0$, there exist $R_i : X\to \ell^1$ and $S_i: \ell^1 \to Y$ with $\|R_i\| \leq 1, \|S_i\| \leq \Fac^\ep_1(W_i) + \delta$ and $\|W_i - S_iR_i\| \leq \ep$.
Define $R: X \to \ell^1\oplus_1\ell^1\oplus_1\ell^1$ and $S: \ell^1\oplus_1\ell^1\oplus_1\ell^1 \to Y$ by $Rx = (R_1x, R_2x, R_3x)$ and $S(z_1,z_2,z_3) = \sum^3_{i=1}S_iz_i$. Then $\|R\|\leq 3$, $\|S\| \leq \max_{1\leq i \leq 3}\Fac^\ep_1(W_i) + \delta$, and $\|W - SR\| \leq 3\ep$. Since $\ell^1\oplus_1\ell^1\oplus_1\ell^1$ is isometric to $\ell^1$, the proof is complete.
\end{proof}

\begin{prop}
Let $X$ be a  Banach space. Take $\G^i_{\ell^1}(X)$ to be the set of operators $T$ in $L(X)$ such that there exists an embedding $\theta: X \to \ell^\infty$ with $\theta T \in \ol{\G_{\ell^1}(X,\ell^\infty)}$. Then $\G^i_{\ell^1}(X)$ is a closed two-sided ideal in $L(X)$.  It is proper provided that $X$ contains uniformly isomorphic copies of $\ell^\infty(k)$. 
\end{prop}

\begin{proof}
Suppose that $\theta$ and $T$ are as stated in the proposition. By the extension property of operators into $\ell^\infty$, for any bounded linear operator $W:X\to X$, there exists a bounded linear operator $V: \ell^\infty\to \ell^\infty$ such that $V\theta = \theta W$.  Thus $\theta WT = V\theta T \in  \ol{\G_{\ell^1}(X,\ell^\infty)}$.  The verification of the remaining conditions for $\G^i_{\ell^1}(X)$ to be a two-sided closed ideal in $L(X)$ is straightforward.  If the identity map on $X$ belongs to  $\G^i_{\ell^1}(X)$ and $X$ contains uniformly isomorphic copies of $\ell^\infty(k)$, then the identity operators $1_k: \ell^\infty(k)\to \ell^\infty(k)$ uniformly factor through $\ell^1$, which is impossible.
\end{proof}

Fix a surjective map $Q: \ell^1\to (\oplus E_n')_{c_0}$ with $\|Q\| \leq 1$.  Then $\theta = Q' : E \to \ell^\infty$ is an embedding such that $\|\theta\|\leq 1$. If $I_1$ and $I_2$ are subsets of $\N$, we write $I_1 < I_2$ to mean that $\max I_1 < \min I_2$.

\begin{prop}\label{prop3}
Let $T$ be an operator in $L(E)\bs \G^i_{\ell^1}(E)$. There exist $\ep > 0$, $n_1 < n_2 < \cdots$, subsets $I_1 < I_2 < \cdots$ of $\N$, such that, taking $W_j = P_{I_j}TJ_{n_j}$, we have $\sup_j\Fac^\ep_1(\theta J_{I_j}W_j) = \infty$.
\end{prop}

\begin{proof}
There exists  $\ep > 0$ so that $d(\theta T, \G_{\ell^1}(E,\ell^\infty)) > 3\ep$. By Lemma \ref{lem1}, $\sup_n\Fac^{3\ep}_1(\theta T_n) = \infty$, where $T_n = TJ_n$.
Let $n_0 = 1$ and $I_0 = \emptyset$. Suppose that $n_0 < \cdots < n_{j-1}$ and $I_0 < \cdots < I_{j-1}$ have been chosen.
Let $r = \max I_{j-1}$.  Then the maps $(P_{[1,r]}T_n)^\infty_{n=1}$ are uniformly bounded and have common finite dimensional range $E_{[1,r]}$.  Therefore, $\sup_n\Fac^\ep_1(P_{[1,r]}T_n) < \infty$. Hence $C = \sup_n\Fac^\ep_1(\theta J_{[1,r]}P_{[1,r]}T_n) < \infty$.  Pick $n_j > n_{j-1}$ so that $\Fac^{3\ep}_1(\theta T_{n_j}) > 3(C+j)$. There exists $s > r$ such that $\|J_{[s,\infty)}P_{[s,\infty)}T_{n_j}\| < \ep$. Then $\Fac^\ep_1(\theta J_{[s,\infty)}P_{[s,\infty)}T_{n_j}) = 0$. Since \[ \theta T_{n_j}  = \theta J_{[1,r]}P_{[1,r]}T_{n_j} + \theta J_{(r,s)}P_{(r,s)}T_{n_j} + \theta J_{[s,\infty)}P_{[s,\infty)}T_{n_j},\]
it follows from Lemma \ref{lem2} that
\[ 3(C+j) < \Fac^{3\ep}_1(\theta T_{n_j}) \leq 3(C \vee \Fac^\ep_1(\theta J_{(r,s)}P_{(r,s)}T_{n_j})).\]
Let $I_j = (r,s)$.  Then $I_j > I_{j-1}$ and $\Fac^\ep_1(\theta J_{I_j}W_{j}) > j$.
\end{proof}

Assume that 
\begin{enumerate}
\item[($*$)]
$(\dim E_n)$ is unbounded and that there is a finite constant $C$ such that for any $m < n$, there are bounded linear operators $i_{mn}: E_m\to E_n$ and $q_{mn}: E_n \to E_m$ so that $q_{mn}i_{mn} = 1_{E_m}$, $\|i_{mn}\| \leq 1$, and $\|q_{mn}\| \leq C$.
\end{enumerate}
Such is the case, in particular, if $E_n = \ell^\infty(n)$ for all $n$.
Given two sequences $(U_n)$ and $(V_n)$ of bounded linear operators $U_n : X_n\to Y_n$, $V_n: G_n \to H_n$, we say that $U_n$ {\em factor through $V_n$ uniformly} if there exists $K < \infty$ such that for any $n$, there exist ${k_n}$ and $R_n \in L(H_{k_n},Y_n)$, $S_n \in L(X_n, G_{k_n})$ satisfying $U_n = R_nV_{k_n}S_n$ and $\|R_n\|\,\|S_n\| \leq K$.

\begin{thm}\label{thm 3.1}
Suppose that $T \in L(E)$ and condition ($*$) holds. Let $I_1 < I_2 < \cdots$, $N_1 < N_2 < \cdots$ be finite subsets of $\N$ and set $W_j = P_{I_j}TJ_{N_j}$, $j \in \N$.  Assume that $1_{E_n}$ factor through $W_j$ uniformly.  Then $1_E$ factors through $T$.
\end{thm}

\begin{proof}
Note that if a subsequence of $1_{E_n}$ factor through $W_j$ uniformly, then so do $1_{E_n}$ by property ($*$). Replace $(W_j)$ by a subsequence if necessary to assume that there is a finite constant $K$ and operators $R_j \in L(E_{I_j},E_j)$, $S_j \in L(E_j, E_{N_j})$ such that $\|R_j\|\leq 1$, $\|S_j\| \leq K$ and $1_{E_j}= R_jW_jS_j$ for all $j$.
Choose $\ep > 0$ so that 
\begin{equation}\label{eq1}
C\bigl\{K\bigl(\sum_{r<s}\frac{\ep}{2^{2r+2s}} + \sum_{2\leq r<s}\frac{\ep}{2^{2r+2s-1}}+\sum^\infty_{r=1}\frac{\ep}{2^{4r-1}}\bigr) + \sum^\infty_{s=1}\frac{\ep}{2^{2s+1}}\bigr\} < 1.
\end{equation}
For each $r$, $\lim_s\|P_{I_s}TJ_{N_r}\| = 0$.  Hence, by using a further subsequence of $(W_j)$,  we may assume that 
\begin{equation}\label{eq1.1}
\|P_{I_s}TJ_{N_r}\| \leq \frac{\ep}{2^{r+s}}\quad \text{if $s > r \geq 1$}.
\end{equation}
For any $p\in \N$, denote by $V_p$ the operator $P_{I_p}TJ_{\cup^\infty_{r=p+1}N_r}: E_{\cup^\infty_{r=p+1}N_r} \to E_{I_p}$.  
Fix $k$ and $p$. If $j > p$, let $e_{jp} = P_{\cup^\infty_{r=p+1}N_r}J_{N_j}$ be the natural embedding of $E_{N_j}$ into $E_{\cup^\infty_{r=p+1}N_r}$.  The sequence 
$V_pe_{jp}S_ji_{kj}$, $j > \max\{k,p\}$, is a uniformly bounded sequence of operators between the finite dimensional spaces $E_k$ and $E_{I_p}$ and hence is compact in the operator norm.  Thus, there are an infinite subsequence $(m_p)$ of $\N$  and a decreasing sequence $(M_p)$ of infinite subsets of $\N$ such that $m_p\in M_p$ and that
\begin{equation}\label{eq2}
\diam\{V_{m_p}e_{jm_p}S_{j}i_{rj}:j\in M_{k+1}\} \leq \frac{\ep}{2^{k+1}}, \quad {1\leq p, r \leq k},
\end{equation}
where the diameter is measured with respect to the operator norm. Define $R, S$ and $q \in L(E)$ by 
\begin{align*} 
Rx &= \sum^\infty_{s=1}J_{m_{2s}}R_{m_{2s}}P_{I_{m_{2s}}}x, \\
Sx  &= \sum^\infty_{r=1}J_{N_{m_{2r}}}S_{m_{2r}}i_{r,m_{2r}}P_rx - \sum^\infty_{r=2}J_{N_{m_{2r-1}}}S_{m_{2r-1}}i_{r,m_{2r-1}}P_rx,\\
qx &= \sum^\infty_{r=1}J_rq_{r,m_{2r}}P_{m_{2r}}x
\end{align*}
for any $x\in E$.  It is easy to check that $R, S ,q$ are indeed bounded linear operators on $E$.
In particular, $\|qx\| \leq C\sum\|P_{m_{2r}}x\| \leq C\|x\|$ and hence $\|q\| \leq C$. 
Since $\|J_I\|, \|P_I\|, \|R_m\|, \|i_{mn}\|\leq 1$ and $\|S_m\| \leq K$,
\begin{align}\label{eq4} 
\sum_{r < s}\|J_{m_{2s}}R_{m_{2s}}P_{I_{m_{2s}}}&TJ_{N_{m_{2r}}}S_{m_{2r}}i_{r,m_{2r}}P_r\| \\ \notag &\leq K\sum_{r < s}\|P_{I_{m_{2s}}}TJ_{N_{m_{2r}}}\| \\ \notag
&\leq K\sum_{r < s}\frac{\ep}{2^{m_{2r}+m_{2s}}} \text{ by (\ref{eq1.1})}\\ \notag
& \leq K\sum_{r < s}\frac{\ep}{2^{2r+2s}}.
\end{align}
Similarly,
\begin{equation}\label{eq5}
\sum_{2\leq r <s}\|J_{m_{2s}}R_{m_{2s}}P_{I_{m_{2s}}}TJ_{N_{m_{2r-1}}}S_{m_{2r-1}}i_{r,m_{2r-1}}P_r\| \leq K\sum_{2 \leq r <s}\frac{\ep}{2^{2s+2r-1}}
\end{equation}
If $r > s$, then $S_{m_{2r}}E_{m_{2r}} \subseteq E_{N_{m_{2r}}}$. 
Since $N_{m_{2r}} \subseteq N = \cup^\infty_{l = m_{2s}+1}N_l$, $J_NP_NJ_{N_{m_{2r}}} = J_{N_{m_{2r}}}$.
Thus,
\[ P_{I_{m_{2s}}}TJ_{N_{m_{2r}}}S_{m_{2r}} = (P_{I_{m_{2s}}}TJ_{N})(P_N J_{N_{m_{2r}}})S_{m_{2r}} = V_{m_{2s}}e_{m_{2r}m_{2s}}S_{m_{2r}}.\]
Similarly,
\[ P_{I_{m_{2s}}}TJ_{N_{m_{2r-1}}}S_{m_{2r-1}} = V_{m_{2s}}e_{m_{2r-1}m_{2s}}S_{m_{2r-1}}.\]
Therefore,
\begin{align}\label{eq6}
\sum_{r>s}\|J_{m_{2s}}&R_{m_{2s}}P_{I_{m_{2s}}}T (J_{N_{m_{2r}}}S_{m_{2r}}i_{r,m_{2r}}P_r - J_{N_{m_{2r-1}}}S_{m_{2r-1}}i_{r,m_{2r-1}}P_r)\|\\
\notag &\leq \sum_{r>s}\|J_{m_{2s}}\|\,\|R_{m_{2s}}\|\,\|V_{m_{2s}}e_{m_{2r}m_{2s}}S_{m_{2r}}i_{r,m_{2r}} -\\ &\notag \quad\quad\quad\quad\quad\quad\quad\quad\quad\quad   V_{m_{2s}}e_{m_{2r-1}m_{2s}}S_{m_{2r-1}}i_{r,m_{2r-1}}\|\,\|P_r\|\\
\notag &\leq   \sum^\infty_{s=1}\sum^\infty_{r=s+1}\diam\{V_{m_{2s}}e_{jm_{2s}}S_{j}i_{rj}:j\in M_{2r-1}\}\\
\notag &\leq \sum^\infty_{s=1}\sum^\infty_{r=s+1}\frac{\ep}{2^{2r-1}} \quad \text{by (\ref{eq2})}\\ &\leq \notag \sum^\infty_{s=1}\frac{\ep}{2^{2s+1}}.
\end{align}
Also,
\begin{align}\label{eq6.1}
\sum^\infty_{r=1}\|J_{m_{2r}}&R_{m_{2r}}P_{I_{m_{2r}}}TJ_{N_{m_{2r-1}}}S_{m_{2r-1}}i_{r,m_{2r-1}}P_r\|\\
\notag&\leq \sum^\infty_{r=1}K\|P_{I_{m_{2r}}}TJ_{N_{m_{2r-1}}}\|\leq K\sum^\infty_{r=1}\frac{\ep}{2^{4r-1}} \quad \text{by (\ref{eq1.1})}.
\end{align}
Furthermore,
\begin{align}\label{eq7}
q\sum^\infty_{r=1} J_{m_{2r}}&R_{m_{2r}}P_{I_{m_{2r}}}TJ_{N_{m_{2r}}}S_{m_{2r}}i_{r,m_{2r}}P_rx\\
\notag &= q\sum^\infty_{r=1}  J_{m_{2r}}R_{m_{2r}}W_{m_{2r}}S_{m_{2r}}i_{r,m_{2r}}P_rx\\
\notag &= q\sum^\infty_{r=1}  J_{m_{2r}}i_{r,m_{2r}}P_rx\\
\notag &= \sum^\infty_{r=1}  J_rq_{r,m_{2r}}i_{r,m_{2r}}P_rx = \sum^\infty_{r=1}  J_rP_rx = x
\end{align}
for all $x\in E$. Combining formulae (\ref{eq4}) to (\ref{eq7}), we find that
\begin{align*} 
\|qRTS - 1_E\| \leq \|q\|\bigl\{K\bigl(\sum_{r<s}\frac{\ep}{2^{2r+2s}} &+ \sum_{2\leq r<s}\frac{\ep}{2^{2r+2s-1}}+ \sum^\infty_{r=1}\frac{\ep}{2^{4r-1}} \bigr)\\ &+ \sum^\infty_{s=1}\frac{\ep}{2^{2s+1}}\bigr\} < 1
\end{align*}
by (\ref{eq1}), since $\|q\|\leq C$.
Hence $qRTS$ is invertible in $L(E)$. Therefore,
\[ 1_E = ((qRTS)^{-1}qR)TS\]
factors through $T$.
\end{proof}

Next, we apply the crucial dichotomy theorem from \cite{LOSZ}.  The precise version we will employ is derived as a corollary.  The closed unit ball of a Banach space $X$ is denoted by $B_X$.

\begin{thm} \label{thm4} \cite[Theorem 2.1]{LOSZ}
Let $U_m:X_m \to L^1$ be a uniformly bounded sequence of operators. Then then the following dichotomy holds:
\begin{enumerate}
\item either the identity operators $1_{\ell^1(k)}$ uniformly factor through the $U_m$, 
\item or, for any $\ep > 0$, there exists $M>0$ such that for all $m$, there exists $g_m \in L^1_+$ with $\|g_m\|_{L^1}\leq M$ and $U_mB_{X_m} \subseteq \{f\in L^1: |f| \leq g_m\} + \ep B_{L^1}$.
\end{enumerate}
\end{thm}

\begin{cor}\label{cor5}
Let $X$ be a Banach space and let $U_m:X \to L^1$ be a uniformly bounded sequence of finite rank operators. Then
\begin{enumerate}
\item either the identity operators $1_{\ell^1(k)}$ uniformly factor through the $U_m$,
\item or, for any quotient map $\xi: \ell^1\to X$, $\sup_m\Fac^\ep_0(U_m\xi) < \infty$ for all $\ep > 0$.
\end{enumerate}
\end{cor}

\begin{proof}
Suppose that condition (b) of Theorem \ref{thm4} holds and let $\xi: \ell^1 \to X$ be a quotient map. We may assume that $\|\xi\| \leq 1$. Denote the unit vector basis of $\ell^1$ by $(e_k)$.
Given $\ep >0$, let $M$ and $g_m$ be as in (b) of Theorem \ref{thm4} for the constant $\ep/2$ in place of $\ep$.
For each $r$, let $\E_r$ denote the conditional expectation operator with respect to the $\sigma$-algebra generated by the dyadic intervals $([\frac{j-1}{2^r},\frac{j}{2^r}))^{2^r}_{j=1}$. $\E_r$ is a positive operator on $L^1$ of norm $1$ and $\lim\E_r = 1_{L^1}$ in the strong operator topology. Since $U_m$ has finite rank, there exists $r_m$ such that $\|U_m - \E_{r_m}U_m\| \leq \ep/2$.
For all $m$ and $k$, there exists $f_{mk} \in L^1$, $|f_{mk}| \leq g_m$ and $\|U_m\xi e_k - f_{mk}\|_{L^1} \leq \ep/2$.
Then $|\E_{r_m}f_{mk}| \leq \E_{r_m} g_m$ and 
\[ \|U_m\xi e_k - \E_{r_m} f_{mk}\| \leq \|U_m\xi e_k - \E_{r_m} U_m\xi e_k\|+  \|\E_{r_m} U_m\xi e_k - \E_{r_m} f_{mk}\| \leq \ep. \]
Take $a_{mk} = (a_{mk}(j))^{2^{r_m}}_{j=1} \in \ell^\infty(2^{r_m})$, $|a_{mk}| \leq 1$, such that 
\[\E_{r_m}f_{mk} = \sum^{2^{r_m}}_{j=1}a_{mk}(j)\chi_{[\frac{j-1}{2^{r_m}},\frac{j}{2^{r_m}})}\cdot \E_{r_m}g_m.\]
Define $S_m: \ell^1\to c_0$ by $S_me_k = (a_{mk}(1),\dots,a_{mk}(2^{r_m}),0,\dots)$, and $R_m: c_0\to L^1$ by $R_m((a_j)) = \sum^{2^{r_m}}_{j=1}a_j\chi_{[\frac{j-1}{2^{r_m}},\frac{j}{2^{r_m}})}\cdot \E_{r_m}g_m$. Then $\|S_m\| \leq 1$, $\|R_m\| \leq M$ and $\|U_m\xi - R_mS_m\| \leq \ep$. This shows that $\sup_m\Fac^\ep_{0}(U_m\xi) \leq M$.
\end{proof}

We are now ready to prove the main result of the paper.

\begin{thm}\label{thm7}
Let $E = (\oplus \ell^\infty(n))_{\ell^1}$.  If $T$ is a bounded linear operator on $E$ such that $T \notin \G^i_{\ell^1}(E)$, then $1_E$ factors through $T$.  Consequently, $\G^i_{\ell^1}(E)$ is the unique maximal ideal of $L(E)$.
\end{thm}

\begin{proof}
It has been observed that $\G^i_{\ell^1}(E)$ is a proper closed ideal of $L(E)$.  Hence the second assertion of the theorem follows from the first. Let $E_n = \ell^\infty(n)$ and $F_n = \ell^1(n)$.  
Set $F = (\oplus F_n)_{c_0}$ and $F_I = (\oplus_{n\in I} F_n)_{c_0}$ for any subset $I$ of $\N$.  Denote by $\pi_I$ the natural projection from $F$ onto $F_I$. Then $F'$ and $F_I'$ may be naturally identified with $E$ and $E_I$ respectively. 
With this identification, $\pi_I' = J_I$.
Recall that we have fixed a quotient map $Q:\ell^1 \to F$ and let $\theta = Q'$.
By Proposition \ref{prop3}, there exist $\ep > 0$, $n_1 < n_2 < \cdots$ and  $I_1 < I_2 < \cdots \subseteq \N$, such that, setting $W_j = P_{I_j}TJ_{n_j}: E_{n_j} \to E_{I_j}$, we have
$\sup_j\Fac^\ep_1(\theta J_{I_j}W_j) = \infty$.  
Since $I_j$ is a finite set, $W_j' \in L(F_{I_j},F_{n_j})$ and $W_j'' = W_j$. Let $V_j = W_j'\pi_{I_j}: F \to F_{n_j}$.
Then $\sup_j\Fac^\ep_1((V_jQ)') = \infty$ and hence $\sup_j\Fac^\ep_0(V_jQ) = \infty$.
Fix bounded linear operators $\al_j: F_{n_j}\to L^1$, $\beta_j:L^1\to F_{n_j}$ such that $\|\al_j\|, \|\beta_j\| = 1$ and $\beta_j\al_j = 1_{F_{n_j}}$.  Let $U_j = \al_jV_j: F \to L^1$.
Since $V_jQ = \beta_jU_jQ$ and $\|\beta_j\| = 1$, we must have $\sup_j\Fac^\ep_0(U_jQ) \geq \sup_j\Fac^\ep_0(V_jQ) = \infty$.
By Corollary \ref{cor5}, $1_{\ell^1(k)}$ uniformly factor through the $U_j$. Thus $1_{\ell^\infty(k)}$ uniformly factor through the $U_j' = J_{I_j}W_j\al_j'$ and hence through the $W_j$.  Therefore, $1_E$ factors through $T$ by Theorem \ref{thm 3.1}.
\end{proof}

\noindent {\bf Remark.} For a sequence of operators $U_m: (\oplus^m_{i=1}H_i)_{\ell^\infty(m)} \to K$, where $H_i$ and $K$ are Hilbert spaces, \cite[Lemma 5.3]{LLR} leads to a dichotomy theorem analogous to Corollary \ref{cor5}.  As a result, one may adapt the foregoing arguments to prove that for $E = (\oplus \ell^2(n))_{\ell^1}$, $\ol{\G_{\ell^1}(E))}$ is the unique maximal ideal of $L(E)$.  The main result of \cite{LSZ} follows easily.


\begin{thebibliography}{99}

\bibitem{AH}\textsc {Spriros A.\ Argyros and Richard G.\ Haydon}, A hereditarily indecomposable ${\mathcal L}_\infty$ -space that solves the scalar-plus-compact problem. Acta Math.\ {\bf 206}(2011), no. 1, 1–54.

\bibitem{C}\textsc{J.W.\ Calkin}, Two-sided ideals and congruences in the ring of bounded operators in Hilbert space, Ann.\ Math.\ {\bf 42}(1941), 839-873.

\bibitem{GMF}\textsc{I.C.\ Gohberg, A.S.\ Markus, I.A.\ Feldman}, Normally solvable operators and ideals associated with them, AMS Translation {\bf 61}(1967), 63-84.

\bibitem{G}\textsc{B.\ Gramsch}, Eine Idealstruktur Banachscher Operatoralbebren, J. Reine Angew.\ Math.\ {\bf 225}(1967), 97-115.




\bibitem{KPSTT}\textsc{Anna Kami\'{n}ska, Alexey I.\ Popov, Eugeniu Spinu, Adi Tcaciuc, Vladimir G.\ Troitsky,} Norm closed operator ideals in Lorentz sequence spaces, J.\ Math.\ Anal.\ Appl.\ {\bf 389} (2012), 247–260.

\bibitem{KL}\textsc{Tomasz Kania, Niels Jakob Laustsen}, Uniqueness of the maximal ideal of the Banach algebra of bounded operators on $C([0,\omega_1])$, J.\ Funct.\ Anal.\ {\bf 262} (2012), 4831–4850. 

\bibitem {LLR}\textsc{N.J.\ Laustsen, R.J.\ Loy and C.J.\ Read}, The lattice of closed ideals in the Banach algebra of operators on certain Banach space, J.\ Funct.\ Anal.\ {\bf 214} (2004), 106-113.

\bibitem {LOSZ}\textsc{N.J.\ Laustsen, E.\ Odell, Th.\ Schlumprecht and A.\ Zs\'{a}k}, Dichotomy theorems for random matrices and closed ideals of operators on $(\oplus^\infty_{n=1}\ell^n_1)_{c_0}$, J.\ Lond.\ Math.\ Soc.\  {\bf 86}(2012), no. 1, 235-258.

\bibitem {LSZ}\textsc{N.J.\ Laustsen, Th.\ Schlumprecht and A.\ Zs\'{a}k}, The lattice of closed ideals in the Banach algebra of operators on a certain dual Banach space, J.\ Operator Theory {\bf 56} (2006), 391-402.

\bibitem {LinSZ}\textsc{Peikee Lin, B\"{u}nyamin Sari, Bentuo Zheng},
Norm closed ideals in the algebra of bounded linear operators on Orlicz sequence spaces, Proc.\ Amer.\ Math.\ Soc., to appear.

\bibitem{L}\textsc{E.\ Luft}, The two-sided closed ideals of the algebra of bounded linear operators of a Hilbert space, Czechoslovak Math.\ J.\ {\bf 18}(1968), 595-605.

\bibitem {SST}\textsc{B.\ Sari, Th.\ Schlumprecht, N.\ Tomczak-Jaegermann, V.G.\ Troitsky}, On norm closed ideals in $L(\ell_p \oplus \ell_q )$, Studia Math. 179 (2007) 239–262.

\bibitem {S}\textsc{Th.\ Schlumprecht}, On the closed subideals of $L(\ell_p \oplus \ell_q )$, Oper.\ Matrices {\bf 6} (2012), 311–326.







\end{thebibliography}
\end{document}